\documentclass[11pt,twoside,letter]{article}
\usepackage{amsthm, amsmath, amssymb, amsfonts, graphicx, epsfig}
\usepackage{accents}

\usepackage{bbm}
\usepackage{mathrsfs}

\usepackage{epstopdf}

\usepackage{graphicx}
\usepackage[font=sl,labelfont=bf]{caption}
\usepackage{subcaption}

\usepackage{float}
\restylefloat{table}

\usepackage{enumerate}

\usepackage[usenames,dvipsnames]{color}

\usepackage[ %
    colorlinks=true, 
    pdfstartview=FitV, 
    linkcolor=blue,
    citecolor=blue, 
    urlcolor=blue]{hyperref}
\usepackage[usenames]{color}

\usepackage{url}
\makeatletter\def\url@leostyle{%
 \@ifundefined{selectfont}{\def\UrlFont{\sf}}{\def\UrlFont{\scriptsize\ttfamily}}} \makeatother

\usepackage[framemethod=default]{mdframed}

\usepackage[margin=1.0in, letterpaper]{geometry}

\newtheorem{theorem}{Theorem}
\newtheorem{proposition}[theorem]{Proposition}
\newtheorem{lemma}[theorem]{Lemma}

\theoremstyle{definition}

\theoremstyle{remark}
\newtheorem{remark}[theorem]{Remark}

\numberwithin{equation}{section}
\numberwithin{theorem}{section}

\definecolor{Red}{rgb}{0.9,0,0.0}
\definecolor{Blue}{rgb}{0,0.0,1.0}

\def\cA{\mathcal{A}}

\def\cP{\mathcal{P}}
\def\cQ{\mathcal{Q}}

\def\cT{\mathcal{T}}

\def\bE{\mathbb{E}}
\def\bF{\mathbb{F}}

\def\bN{\mathbb{N}}

\def\bP{\mathbb{P}}
\def\bQ{\mathbb{Q}}
\def\bR{\mathbb{R}}

\def\sF{\mathscr{F}}

\newcommand{\set}[1]{\{#1\}}            
\renewcommand{\mid}{\;|\;}              

\DeclareMathOperator{\var}{\mathrm{V}@\mathrm{R}}           
\DeclareMathOperator{\glr}{\mathrm{GLR}}                    

\floatstyle{plaintop}

\title{ \vspace{-3em} 
    Adaptive Robust Control Under Model Uncertainty
}

\makeatletter
\def\and{%
  \end{tabular}%
  \begin{tabular}[t]{c}}%
\def\@fnsymbol#1{\ensuremath{\ifcase#1\or a\or b\or c\or
   d\or e\or f\or g\or h\or i\else\@ctrerr\fi}}
\makeatother

\author{
        Tomasz R. Bielecki\,\thanks{Department of Applied Mathematics, Illinois Institute of Technology
       \newline \hspace*{1.45em}  Emails: \url{tbielecki@iit.edu} (T. R. Bielecki), \url{tchen29@iit.edu} (T. Chen) and \url{cialenco@iit.edu} (I. Cialenco)
       \newline \hspace*{1.45em}  URLs: \url{http://math.iit.edu/\~bielecki}  and \url{http://math.iit.edu/\~igor}
        \vspace{0.5em}} ,
\and
        Tao Chen,\,\footnotemark[1] \newline
\and
         Igor Cialenco,\,\footnotemark[1] \newline
\and
        Areski Cousin\,\thanks{
        Institut de Sciences Financi\`{e}re et d'Assurance, Universit\'{e} Lyon 1, Lyon, France
         \newline \hspace*{1.45em}  Email: \url{areski.cousin@univ-lyon1.fr},  URL: \url{http://www.acousin.net/}
         \vspace{0.5em}} ,
\and
        Monique Jeanblanc\,\thanks{
        Univ Evry, LaMME UMR CNRS 807,  Universit\'e Paris-Saclay, 91025, Evry, France
         \newline \hspace*{1.45em}  Email: \url{monique.jeanblanc@univ-evry.fr},
         URL: \url{http://www.maths.univ-evry.fr/pages_perso/jeanblanc}
         }
        }

\date{ {\small 
First Circulated: June 06, 2017
}}

\begin{document}

\maketitle

{\footnotesize
\begin{tabular}{l@{} p{350pt}}
  \hline \\[-.2em]
  \textsc{Abstract}: \ & In this paper we propose a new methodology for solving an uncertain stochastic Markovian control problem in discrete time. We call the proposed methodology the adaptive robust control. We demonstrate that the uncertain control problem under consideration can be solved in terms of associated adaptive robust Bellman equation. The success of our approach is to the great extend owed to the recursive methodology for construction of relevant confidence regions. We illustrate our methodology by considering an optimal portfolio allocation problem, and we compare results obtained using the adaptive robust control method with some other existing methods. \\[1em]
\textsc{Keywords:} \ & adaptive robust control, model uncertainty, stochastic control, dynamic programming, recursive confidence regions, Markov control problem, portfolio allocation.
 \\
\textsc{MSC2010:} \ & 93E20, 93E35, 49L20, 60J05  \\[1em]
  \hline
\end{tabular}
}


\section{Introduction}
In this paper we propose a new methodology for solving an uncertain stochastic Markovian control problem in discrete time, which is then applied to an optimal portfolio selection problem. Accordingly, we only consider terminal optimization criterion.

The uncertainty in the problem comes from the fact that the (true) law of the underlying stochastic process is not known. What is known, we assume, is the family of potential probability laws that the true law belongs to. Thus, we are dealing here with a stochastic control problem subject to Knightian uncertainty.

Such problems have been extensively studied in the literature, using different approaches, some of them  are briefly described in  Section~\ref{sec:ocsmu}.

The classical approach to the problem goes back to several authors. The multiple-priors, or the maxmin approach, of \cite{GilboaSchmeidler1989} is probably one of the first ones and one of the best-known in the economics literature. In the context of our terminal criterion, it essentially amounts to a robust control problem (a game problem, in fact) of the form
\[
\sup_{\varphi\in \mathcal{A}}\inf_{\bQ\in \mathcal{M}}\bE_\bQ\, \mathcal{L}(X,\varphi,T),
\]
where $\varphi$ is the control process,  $\mathcal{A}$ is the family of admissible controls, $\mathcal{M}$ is a family of possible underlying probabilistic models (or priors),  $\bE_\bQ$ denotes expectation under the prior $\bQ$, $X$ is the underlying process, $T$ is the finite optimization horizon, and where $\mathcal{L}$ is an optimization criterion. The family $\mathcal{M}$ represents the Knightian uncertainty.

The above maxmin formulation has been further modified to the effect of
\[
\sup_{\varphi\in \mathcal{A}}\inf_{Q\in \mathcal{M}}\bE_\bQ \left (\mathcal{L}(X,\varphi,T)-c(\bQ)\right ),
\]
where $c$ is a penalty function. We refer to, e.g.,   \cite{HansenSargent2006}, \cite{Skiadas2003}, \cite{MaccheroniMarinacciRustichini2006} and \cite{BordigoniMatoussiSchweizer2007}, and references therein,  for discussion and various studies of this problem.

In our approach we do not use the penalty term. Instead, we apply a learning algorithm that is meant to reduce the uncertainty about the true probabilistic structure underlying the evolution of the process $X$. This leads us to consider what we call \textit{adaptive robust control problem}. We stress that our problem and  approach should not be confused with the problems and approaches presented in \cite{IoannouSunBook2012} and \cite{BertuccelliWuHow2012}.

A very interesting study of an uncertain control problem in continuous time, that involves learning,  has been done in \cite{KallbladOblojTZ2014}.

The paper is organized as follows. In Section~\ref{sec:ocsmu} we briefly review some of the existing methodologies of solving stochastic control problems subject to model uncertainty, starting with robust control method, and continuing with strong robust method, model free robust method, Bayesian  adaptive control method, and adaptive control method. Also here we introduce the underlying idea of the proposed method, called the adaptive robust control, and its relationship with the existing ones. Section~\ref{sec:robust} is dedicated to the adaptive robust control methodology. We begin with setting up the model and in Section~\ref{sec:formulation} we formulate in strict terms the stochastic control problem. The solution of the adaptive robust problem is discussed in Section~\ref{sec:SolutionAdapRobust}. Also in this section, we derive the associated Bellman equation and prove the Bellman principle of optimality for the considered  adaptive robust problem. Finally, in Section~\ref{sec:examples} we consider an illustrative example, namely, the classical dynamic optimal allocation problem when the investor is deciding at each time on investing in a risky asset and a risk-free banking account by maximizing the expected utility of the terminal wealth. Also here, we give a comparative analysis between the proposed method and some of the existing classical methods.

\section{Stochastic Control Subject to Model Uncertainty}\label{sec:ocsmu}

Let  $(\Omega, \sF)$ be a measurable space,  and $T\in \bN$ be a fixed time horizon. Let $\cT=\set{0,1,2,\ldots,T}$, $\cT'=\set{0,1,2,\ldots,T-1}$, and $\boldsymbol \Theta\subset \bR^d$   be a non-empty set, which will play the role of the known parameter space throughout.\footnote{In general, the parameter space may be infinite dimensional, consisting for example of dynamic factors, such as deterministic functions of time or hidden Markov chains. In this study, for simplicity, we chose the parameter space to be a subset of $\bR^d$. In most applications, in order to avoid problems with constrained estimation, the parameter space is taken to be equal to the maximal relevant
subset of $\bR^d$. }

On the space $(\Omega, \sF)$ we consider a random process $X=\{X_t,\ t\in \cT\}$ taking values in some measurable space. We postulate that this process is observed, and we denote by ${\mathbb {F}}=(\sF_t,t\in \cT)$ its natural filtration. The true law of $X$  is unknown and assumed to be generated by a probability measure belonging to a parameterized family of probability distributions on $(\Omega, \sF)$, say $\mathbf{P}(\boldsymbol \Theta)=\{\bP_\theta,  \theta\in \boldsymbol \Theta\}$.   We will write $\bE_\bP$ to denote the expectation corresponding to a probability measure $\bP$ on  $(\Omega, \sF)$, and, for simplicity, we denote by $\bE_\theta$ the expectation operator corresponding to  the probability $\bP_\theta$.

By $\bP_{\theta^*}$ we denote the measure generating the true law of $X$, so that $\theta ^*\in \boldsymbol
\Theta$ is the (unknown) true parameter. Since $\boldsymbol \Theta$ is assumed to be known, the model uncertainty discussed in this paper occurs only if  $\boldsymbol \Theta \ne \{\theta^*\},$ which we
assume to be the case.

The methodology proposed  in this paper is motivated by the following generic optimization problem: we consider a family, say $\cA$, of ${\mathbb {F}}$--adapted processes $\varphi=\{ \varphi_t,\ t\in \cT\}$ defined on $(\Omega,\sF)$, that take values in a measurable space. We refer to the elements of $\cA$ as to admissible control processes. Additionally, we consider a functional of $X$ and $\varphi$, which we denote by $L.$ The stochastic control problem at hand is
\begin{equation}\label{eq:gen}
\inf_{\varphi\in \cA} \bE_{\theta^*}\left( L(X,\varphi)\right ).
\end{equation}
However, stated as such, the problem can not be dealt with directly, since the value of $\theta^*$ is unknown. Because of this we refer to such problem as to an \textit{uncertain stochastic control problem.} The question  is then how to handle the stochastic control problem \eqref{eq:gen}  subject to this type of model uncertainty.

The classical approaches to solving this uncertain stochastic control problem are:
\begin{itemize}
\item \textit{to solve the robust control problem}
\begin{align}\label{eq:rob}
  \inf_{\varphi\in\cA} \sup_{\theta\in \boldsymbol \Theta} \bE_{\theta} \left( L(X,\varphi)\right ).
\end{align}
We refer to, e.g., \cite{HansenSargent2006},  \cite{HansenBookBook2008}, \cite{BasarBernhardBook1995}, for more information regarding robust control problems.

\item \textit{to solve the strong robust control problem}
\begin{align}\label{eq:rob2}
  \inf_{\varphi\in\cA} \sup_{\mathbb Q\in {\mathcal Q}^{\varphi,{\boldsymbol \Theta}^K}_\nu} \bE_{\mathbb Q} \left( L(X,\varphi)\right ),
\end{align}
where   ${\boldsymbol \Theta}^K$ is the set of strategies chosen by a Knightian adversary (the nature) and ${\mathcal Q}^{\varphi,{\boldsymbol \Theta}^K}_\nu$ a set of probabilities,  depending on the strategy $\varphi$ and a given law $\nu$ on $\boldsymbol \Theta$. See Section~\ref{sec:robust} for a formal description of this problem, and see, e.g., \cite{Sirbu2014} and  \cite{BayraktarCossoPham2014} for related work.

\item \textit{to solve the model free robust control problem}
\begin{align}\label{eq:free}
  \inf_{\varphi\in\cA} \sup_{\bP\in \cP} \bE_{\bP} \left( L(X,\varphi)\right ),
\end{align}
where $\cP$ is given family of probability measures on $(\Omega,\sF).$

\item \textit{to solve a Bayesian adaptive control problem,} where  it is assumed that the (unknown) parameter $\theta$ is random, modeled as a random variable $\Theta$ taking values in $\boldsymbol \Theta$ and having a prior distribution denoted by $\nu_0$.
In this framework, the uncertain control problem  is solved via the following optimization problem
\begin{align*}
  \inf_{\varphi\in\cA} \int_{\boldsymbol \Theta} \bE_\theta  \left(  L(X,\varphi)  \right) \nu_0(d\theta).
\end{align*}
We refer to, e.g., \cite{KumarVaraiya2015Book}.

\item \textit{to solve an adaptive control problem,} namely, first for each $\theta \in \boldsymbol \Theta$ solve
\begin{equation}\label{eq:acp}
\inf_{\varphi\in \cA} \bE_\theta\left( L(X,\varphi)\right ),
\end{equation}
and denote by $\varphi ^\theta$ a corresponding optimal control (assumed to exist). Then, at each time $t\in \cT'$, compute a point estimate $\widehat \theta_t$ of $\theta ^*$, using a chosen, $\sF_t$ measurable estimator {$\widehat \Theta _t$}. Finally, apply at time $t$ the control value $\varphi^{\widehat \theta_t}_t$. We refer to, e.g., \cite{KumarVaraiya2015Book}, \cite{ChenGuo1991-Book}.
\end{itemize}

Several comments are now in order:
\begin{enumerate}
\item Regarding the solution of the robust control problem, \cite{Lim2006}  observe that
\begin{mdframed}[leftmargin=5pt, innerleftmargin = 30pt, hidealllines = true]
\textit{If the true model\footnote{True model in \cite{Lim2006} corresponds to  $\theta ^*$ in our notation.} is the worst one, then this solution will be nice and dandy. However, if the true model is the best one or something close to it, this solution could be very bad (that is, the solution need not be robust to model error at all!).}
\end{mdframed}
The message is that using the robust control framework may produce undesirable results.

\item It  can be shown that
\begin{equation}\label{eq:robadap}
\inf_{\varphi\in\cA } \sup_{\theta\in \boldsymbol \Theta} \bE_\theta\left(  L(S,\varphi)\right )    ) = \inf_{\varphi\in\cA }\sup_{\nu_0\in\cP(\boldsymbol \Theta)} \int_{\boldsymbol \Theta} \bE_\theta\left(  L(S,\varphi)    \right ) \nu_0(d\theta).
\end{equation}

Thus, for any given prior distribution $\nu_0$
$$
\inf_{\varphi\in\cA } \sup_{\theta\in \boldsymbol \Theta} \bE_\theta\left(  L(S,\varphi)   \right ) \geq \inf_{\varphi\in\cA } \int_{\boldsymbol \Theta} \bE_\theta\left(  L(S,\varphi)   \right ) \nu_0(d\theta).
$$
The adaptive Bayesian problem appears to be less conservative. Thus, in principle, solving the adaptive Bayesian control problem for a given prior distribution may lead to a better solution of the  uncertain stochastic control problem than solving the robust control problem.

\item  It is sometimes suggested that the robust control problem does not involve learning about the unknown parameter $\theta^*$, which in fact is the case, but that the adaptive Bayesian control problem involves ``learning'' about $\theta^*$. The reason for this latter claim is that in the adaptive Bayesian control approach, in accordance with the Bayesian statistics, the unknown parameter is considered to be a random variable, say $\Theta$     with prior distribution $\nu_0.$ This random variable is then considered to be an unobserved state variable, and consequently the adaptive Bayesian control problem  is regarded as a control problem with partial (incomplete) observation of the entire state vector. The typical way to solve a control problem with partial observation is by means of transforming it to the corresponding separated control problem. The separated problem is a problem with full observation, which is achieved by introducing additional state variable, and what in the Bayesian statistics is known as  the posterior distribution of $\Theta$. The ``learning'' is attributed to the use of the posterior distribution of  $\Theta$ in the separated problem. However, the information that is used in the separated problem is exactly the same that in the original problem, no learning is really involved. This is further documented by the equality \eqref{eq:robadap}.

\item We refer to Remark \ref{rem-strong-robust} for a discussion regarding distinction between the strong robust control problem \ref{eq:rob2} and the adaptive robust control problem \ref{eq:prob1R-adaptive-0} that is stated below.

\item As said before, the model uncertainty discussed in this paper occurs if  $\boldsymbol \Theta \ne \{\theta^*\}.$ The classical robust control problem \eqref{eq:rob} does not involve any reduction of uncertainty about $\theta^*,$ as the parameter space is not ``updated'' with  incoming information about the signal process $X$. Analogous remark applies to problems \eqref{eq:rob2} and \eqref{eq:free}: there is no reduction of uncertainty about the underlying stochastic dynamics involved there.

    Clearly, incorporating ``learning'' into the robust control paradigm appears like a good idea.
    In fact, in \cite{AndersonHansenSargent2003} the authors state

\begin{mdframed}[leftmargin=10pt, innerleftmargin = 30pt, hidealllines = true]
\textit{We see three important extensions to our current investigation. Like builders of rational expectations models, we have side-stepped the issue of how decision-makers select an approximating model. Following the literature on robust control, we envision this approximating model to be analytically tractable, yet to be regarded by the decision maker as not providing a correct model of the evolution of the state vector. The misspecifications we have in mind are small in a statistical sense but can otherwise be quite diverse. Just as we have not formally modelled how agents learned the approximating model, neither have we formally justified why they do not bother to learn about potentially complicated misspecifications of that model. \textbf{Incorporating forms of learning would be an important extension of our work.}\footnote{The boldface emphasis  is ours.}}
\end{mdframed}

\end{enumerate}

In the present work we follow up on the suggestion of Anderson, Hansen and Sargent stated above, and we propose a new methodology, which we call  \textbf{adaptive robust control methodology}, and which is meant to incorporate learning about $\theta^*$ into the robust control paradigm.

This methodology amounts to solving the following problem
\begin{align}\label{eq:prob1R-adaptive-0}
  \inf_{\varphi\in\cA} \sup_{\mathbb{Q}\in {\mathcal Q}^{\varphi,{\boldsymbol \Psi}}_\nu}  \bE_{\mathbb{Q}}\left( L(S,\varphi)\ \right),
\end{align}
where ${\mathcal Q}^{\varphi,{\boldsymbol \Psi}}_\nu$ is a family of probability measures on some canonical space related to the  process $X$, chosen in a way that allows for appropriate dynamic reduction of uncertainty about $\theta^*$.  Specifically, we chose the family ${\mathcal Q}^{\varphi,{\boldsymbol \Psi}}_\nu$ in terms of confidence regions for the parameter $\theta^*$ (see details in Section \ref{sec:robust}). Thus, the adaptive robust control methodology incorporates updating controller's knowledge about the parameter space -- a form of learning, directed towards reducing uncertainty about $\theta^*$ using  incoming information about the signal process $S$. Problem \eqref{eq:prob1R-adaptive-0} is derived from problem \eqref{eq:rob2}; we refer to Section \ref{sec:robust} for derivation of both problems.

In this paper we will compare the robust control methodology, the adaptive control methodology and the adaptive robust control methodology in the context of a specific optimal portfolio selection problem that is considered in finance. This will be done in Section~\ref{sec:examples}.

\section{Adaptive Robust Control Methodology}\label{sec:robust}
This section is the key section of the paper. We will first make precise the control problem that we are studying and then we will proceed to the presentation of our adaptive robust control methodology.

Let  $\{(\Omega, \sF, \bP_\theta),\ \theta\in \boldsymbol \Theta \subset \bR^d\}$ be a family of probability spaces, and let $T\in \bN$ be a fixed maturity time.  In addition, we let $A\subset \bR^k$ be a finite\footnote{$A$ will represent the set of control values, and we assume it is finite for simplicity, in order to avoid technical issues regarding existence of measurable selectors.} set and
$$
S\, :\, \bR^n\times A\times \bR^m\rightarrow \bR^n
$$
be a measurable mapping. Finally, let
$$
\ell \, :\, \bR^n \rightarrow \bR
$$ be a measurable function.

We consider an  underlying discrete time controlled dynamical system with state process $X$ taking values in $\bR^n$ and control process $\varphi$ taking values in $A$. Specifically, we let
\begin{equation}\label{eq:mm}
X_{t+1}=S(X_t,\varphi_t,Z_{t+1}),\quad t\in\cT',\quad X_0=x_0 \in \bR^n,
\end{equation}
where $Z=\set{{Z}_t,\, t\in\cT'}$ is an $\bR^m$-valued random sequence, which  is $\bF$-adapted and i.i.d. under each measure $\bP_\theta.$\footnote{The assumption that the sequence $Z$ is i.i.d. under each measure $\bP_\theta$ is made in order to simplify our study.}  The true, but unknown law of $Z$ corresponds to measure $\bP_{\theta^*}$. A control process $\varphi$ is admissible, if it is $\bF$-adapted. We denote by $\cA$ the set of all admissible controls.

Using the notation  of Section \ref{sec:ocsmu} we set $L(X,\varphi)= \ell(X_T)$, so that the problem \eqref{eq:gen} becomes now
\begin{equation}\label{eq:gen-1}
\inf_{\varphi\in \cA} \bE_{\theta^*} \ell(X_T).
\end{equation}

\subsection{Formulation of the adaptive robust control problem}\label{sec:formulation}
In what follows, we will be making use of a recursive construction of confidence regions for the unknown parameter $\theta^*$ in our model. We refer to \cite{BCC2016} for a general study of recursive constructions of confidence regions for time homogeneous Markov chains, and to Section~\ref{sec:examples} for details of a specific  recursive construction corresponding to the optimal portfolio selection problem. Here, we just postulate that the recursive algorithm for building confidence regions uses an $\bR^d$-valued and observed  process, say $(C_t,\ t\in \cT')$, satisfying the following abstract dynamics
\begin{equation}\label{eq:R}
C_{t+1}= R(t,C_t,Z_{t+1}),\quad t\in\cT',\ C_0=c_0,
\end{equation}
where $R(t,c,z)$ is a deterministic measurable function. Note that, given our assumptions about process $Z$, the process $C$ is $\bF$-adapted.
This is   one of the key features of our model. Usually $C_t$ is taken to be a consistent estimator of $\theta^*$.

Now, we fix a confidence level $\alpha\in (0,1),$ and for each time $t\in \cT'$, we assume that  an (1-$\alpha$)-confidence region, say $\mathbf{\Theta}_t$, for $\theta^*$, can be represented as
\begin{equation}\label{eq:CIR}
\mathbf{\Theta}_t=\tau(t,C_t),
\end{equation}
where, for each $t\in \cT'$,
\[
\tau(t,\cdot)\, :\, \bR^d \rightarrow 2^{\mathbf{\Theta}}
\]
is a deterministic set valued   function.\footnote{As usual, $2^{\mathbf{\Theta}}$ denotes the set of all subsets of ${\mathbf{\Theta}}$.} Note that in view of \eqref{eq:R} the construction of confidence regions given in \eqref{eq:CIR} is indeed recursive. In our construction of confidence regions, the mapping $\tau(t,\cdot)$ will be a measurable set valued function,  with compact values. The important property of the recursive confidence regions constructed in Section \ref{sec:examples} is that $\lim_{t\rightarrow \infty} \mathbf{\Theta}_t=\set{\theta^*}$, where the convergence is understood $\bP_{\theta^*}$ almost surely, and the limit is in the Hausdorff metric. This is not always the case though in general. In \cite{BCC2016} is shown that the convergence holds in probability, for the model set-up studied there. The sequence $\mathbf{\Theta}_t,\ t\in \cT'$ represents learning about $\theta^*$ based on the observation of the history $H_t,\ t\in \cT$ (cf. \eqref{hist} below). We introduce the augmented state process $Y_t=(X_t,C_t),\ t\in\cT,$  and the augmented state space
\[
E_Y=\bR^n \times \bR^d.
\]
We denote by ${\mathcal E}_Y$ the collection of Borel measurable sets in $E_Y$. The process $Y$ has the following dynamics,
\[
Y_{t+1}=\mathbf{T}(t, Y_t,\varphi_t,Z_{t+1}),\ t\in \cT',
\]
where  $\mathbf{T}$  is the mapping  $$ \mathbf{T}\, :\, \cT'\times E_Y \times A \times \bR^m \rightarrow    E_Y$$ defined as
\begin{equation}\label{eq:T}
\mathbf{T}(t,y,a,z)=\big(S(x,a,z), R(t,c,z)\big),
\end{equation}
where $ y=(x,c)\in  E_Y$.

For future reference, we define the corresponding histories
\begin{equation}\label{hist}
H_t=((X_0,C_0),(X_1,C_1),\ldots,(X_t,C_t)),\ t\in \cT,
\end{equation}
so that
\begin{equation}\label{eq:boldh}
H_t\in \mathbf{H}_t= \underbrace{E_Y \times E_Y \times \ldots
\times E_Y}_{t+1 \textrm{ times}}.
\end{equation}
Clearly, for any admissible control process $\varphi$, the random variable $H_t$ is $\sF_t$-measurable. We denote by
\begin{equation}\label{eq:h}
h_t=(y_0,y_1,\ldots,y_t)=(x_0,c_0,x_1,c_1,\ldots,x_t,c_t)
\end{equation}
 a realization of $H_t.$ Note that $h_0=y_0$.

\begin{remark}
 A control process $\varphi=(\varphi_t,\ t \in \cT')$ is called  history dependent control process if (with a slight abuse of notation)
\[
\varphi_t = \varphi_t(H_t) ,
\]
where (on the right hand side) $\varphi_t \, :\, \mathbf{H}_t \rightarrow A$, is a measurable mapping. Note that any admissible control process $\varphi$ is such that $\varphi_t$ is a function of $X_0,\ldots,X_t.$ So, any admissible control process is history dependent. On the other hand, given our above set up, any history dependent control process is ${\mathbb {F}}$--adapted, and thus, it is admissible. From now on, we identify the set $\mathcal{A}$ of admissible strategies with the set of history dependent strategies.
\end{remark}

For the future reference, for any admissible control process $\varphi$ and for any $t\in\cT'$, we denote by $\varphi^t=(\varphi_k,\ k=t,\dots, T-1)$ the ``$t$-tail'' of $\varphi$; in particular, $\varphi^0=\varphi$. Accordingly, we denote by ${\mathcal A}^t$ the collection of $t$-tails $\varphi^t$;  in particular, ${\mathcal A}^0={\mathcal A}$.

Let $\psi_t:\mathbf{H}_t\to\mathbf{\Theta}$ be a Borel measurable mapping (Knightian selector), and let us denote by $\psi =(\psi_t, \ t\in\cT')$ the sequence of such mappings, and by $\psi^t = (\psi_s, \ s=t,\ldots,T-1)$ the $t$-tail of the sequence $\psi$. The set of all sequences $\psi$, and respectively $\psi^t$, will be denoted by $\mathbf{\Psi}_K$ and $\mathbf{\Psi}_K^t$, respectively.
Similarly, we consider the measurable selectors ${\psi}_t(\cdot):\mathbf{H}_t \to\mathbf{\Theta}_t$, and correspondingly define the set of all sequences of such selectors by $\mathbf{\Psi}$, and the set of $t$-tails by $\mathbf{\Psi}^t$. Clearly, $\psi^t\in\mathbf{\Psi}^t$ if and only if $\psi^t\in\mathbf{\Psi}^t_K$ and $\psi_s(h_s)\in\tau(s,c_s), \ s=t,\ldots,T-1$.

Next, for each $(t,y,a,\theta)\in  \cT'\times  E_Y\times A\times \mathbf{\Theta}$,  we define a probability measure on $\mathcal{E}_Y$:
\begin{equation}\label{eq:QB-bar}
 Q(B\mid t,y,a,\theta)=\bP_\theta(Z_{t+1}\in \{z: \mathbf{T}(t,y,a,z)\in B\})=\bP_\theta\left(\mathbf{T}(t,y,a,Z_{t+1})\in B\right),\ B\in \mathcal{E}_Y.
\end{equation}
We assume that for each $B$ the function $Q(B\mid t,y,a,\theta)$ of $t,y,a,\theta$ is measurable. This assumption will be satisfied in the context of the optimal portfolio problem discussed in Section \ref{sec:examples}.

Finally, using Ionescu-Tulcea theorem, for every control process $\varphi \in \cA $ and for every initial probability distribution $\nu$ on $E_Y$,  we define the family ${\mathcal Q}^{\varphi,{\boldsymbol \Psi}}_\nu =\{{\mathbb Q}^{\varphi,\psi}_\nu,\  \boldsymbol\psi \in {\boldsymbol \Psi} \}$  of probability measures on the  canonical space $E_Y^{T+1}$, with ${\mathbb Q}^{\varphi,{\psi}}_\nu $ given as follows
\begin{align}\label{eq:prob}
\mathbb{Q}^{\varphi,{\psi}}_\nu(B_0,B_1,\ldots,B_T)
 =\int\limits_{B_0}\int\limits_{B_1}\cdots \int\limits_{B_T} \prod_{t=1}^{T} Q( dy_t |  t-1,y_{t-1},\varphi_{t-1}(h_{t-1}),{\psi}_{t-1}(h_{t-1})) \nu(dy_0)
\end{align}
Analogously we define the set ${\mathcal Q}^{\varphi,{\boldsymbol \Psi}_K}_\nu =\{{\mathbb Q}^{\varphi,{\psi}^K}_\nu,\ \psi^K\in {\boldsymbol \Psi}_K \}$.

The \textit{strong robust hedging} problem is then given as:
\begin{align}\label{prob1R-strong}
  \inf_{\varphi\in\cA} \sup_{\mathbb{Q}\in {\mathcal Q}^{\varphi,{\boldsymbol \Psi}_K}_\nu} \bE_{\mathbb{Q}} \ell(X_T).
\end{align}
The corresponding   \textit{adaptive robust hedging} problem is:
\begin{align}\label{prob1R-adaptive}
  \inf_{\varphi\in\cA} \sup_{\mathbb{Q}\in {\mathcal Q}^{\varphi,{\boldsymbol \Psi}}_\nu} \bE_{\mathbb{Q}} \ell(X_T).
\end{align}

\begin{remark}\label{rem-strong-robust}
The strong robust hedging problem is essentially a game problem between the hedger and his/her Knightian adversary -- the nature, who may keep changing the dynamics of the underlying stochastic system over time.  In this game, the nature is not restricted in its choices of model dynamics, except for the requirement that ${\psi}^K_t(H_t) \in \mathbf{\Theta}$, and each choice is potentially based on the entire history $H_t$ up to time $t$. On the other hand, the adaptive robust hedging problem is a game problem between the hedger and his/her Knightian adversary -- the nature, who, as in the case of strong robust hedging problem, may keep changing the dynamics of the underlying stochastic system over time. However, in this game, the nature is restricted in its choices of model dynamics to the effect that ${\psi}_t(H_t)\in \tau(t,C_t)$.
\end{remark}

Note that if the parameter $\theta^*$ is known, then, using the above notations and the canonical construction\footnote{Which, clearly, is not needed in this case.}, the hedging problem reduces to
\begin{align}\label{eq:prob1R-true}
  \inf_{\varphi\in\cA} \bE_{\mathbb{Q}^*}\ell(X_T),
\end{align}
where, formally,  the probability $\mathbb{Q}^*$ is given as in \eqref{eq:prob} with $\tau(t,c)=\{\theta^*\}$ for all $t$ and $c$. It certainly holds that
\begin{align}\label{comp1}
   \inf_{\varphi\in\cA} \bE_{\mathbb{Q}^*}\ell(X_T)  \leq \inf_{\varphi\in\cA} \sup_{\mathbb{Q}\in {\mathcal Q}^{\varphi,{\boldsymbol \Psi}}_\nu} \bE_{\mathbb{Q}}\ell(X_T)  \leq \inf_{\varphi\in\cA} \sup_{\mathbb{Q}\in {\mathcal Q}^{\varphi,{\boldsymbol \Psi}_K}_\nu} \bE_{\mathbb{Q}}\ell(X_T).
\end{align}
It also holds that
\begin{align}\label{eq:comp11}
\inf_{\varphi\in\cA}\sup_{\theta \in {\mathbf \Theta}} \bE_\theta \ell(X_T) \leq \inf_{\varphi\in\cA} \sup_{\mathbb{Q}\in {\mathcal Q}^{\varphi,{\boldsymbol \Psi}_K}_\nu} \bE_{\mathbb{Q}}\ell(X_T).
\end{align}

\begin{remark}
We conjecture that
$$
\inf_{\varphi\in\cA} \sup_{\mathbb{Q}\in {\mathcal Q}^{\varphi,{\boldsymbol \Psi}}_\nu} \bE^{\mathbb{Q}}\ell(X_T) \leq \inf_{\varphi\in\cA}\sup_{\theta \in {\mathbf \Theta}} \bE_\theta \ell(X_T).
$$
However, at this time, we do not know how to prove this conjecture, and whether it is true in general.
\end{remark}

\subsection{Solution of the adaptive robust control problem}\label{sec:SolutionAdapRobust}

In accordance with our original set-up,  in what follows we assume that $\nu(dx_0)=\delta_{h_0}(dx_0)$ (Dirac measure), and we use the notation ${\mathbb Q}^{\varphi,{\psi}}_{h_0}$ and ${\mathcal Q}^{\varphi,{\boldsymbol \Psi}}_{h_0}$ in place of ${\mathbb Q}^{\varphi,{\psi}}_{\nu}$ and ${\mathcal Q}^{\varphi,{\boldsymbol \Psi}}_\nu$, so that the problem \eqref{prob1R-adaptive} becomes
\begin{align}\label{eq:prob1R-adaptive-h0}
  \inf_{\varphi\in\cA} \sup_{\mathbb{Q}\in {\mathcal Q}^{\varphi,{\boldsymbol \Psi}}_{h_0}} \bE_{\mathbb{Q}}\ell(X_T).
\end{align}
For each $t\in \cT'$,  we then define a probability measure on the concatenated canonical space as follows
\begin{align*}
\mathbb{Q}^{\varphi^t, {\psi}^t}_{h_t}(B_{t+1},\ldots,B_T) =
\int\limits_{B_{t+1}}\cdots \int\limits_{B_T}
\prod\limits_{u=t+1}^{T} Q( dy_u\mid u-1,y_{u-1},\varphi_{u-1}(h_{u-1}),{\psi}_{u-1}(h_{u-1})).
\end{align*}
Accordingly, we put ${\mathcal Q}^{\varphi^t, {\boldsymbol \Psi}^t}_{h_t} = \{{\mathbb Q}^{\varphi^{t},{\psi}^t}_{h_{t}},\ \psi^{t}\in {\boldsymbol \Psi}^t \}$. Finally, we define the functions $U_t$ and $U^*_t$ as follows:
for $\varphi^t\in {\mathcal A}^t$ and $h_t  \in H_t$
\begin{align}\label{eq:prob1R-adaptive-Dirac}
   U_t(\varphi^t,h_t)&=\sup_{\mathbb{Q}\in {\mathcal Q}^{\varphi^t, {\boldsymbol \Psi}^t}_{h_t}} \bE_{\mathbb{Q}}
    \ell(X_T),\ t\in \cT',\\
    U^*_t(h_t)&=\inf_{\varphi^t\in {\mathcal A}^t} U_t(\varphi^t,h_t),\ t\in \cT',\\
    U^*_T(h_T) & = \ell(x_T).
\end{align}
Note in particular that
\[
U^*_0(y_0)=U^*_0(h_0)=\inf_{\varphi\in\cA} \sup_{\mathbb{Q}\in {\mathcal Q}^{\varphi,{\boldsymbol \Psi}}_{h_0}} \bE_\bQ \ell(X_T).
\]
We call $U^*_t$ the \textit{adaptive robust Bellman} functions.

\subsubsection{Adaptive robust Bellman equation}
Here we will show that a solution to the optimal problem \eqref{eq:prob1R-adaptive-h0} can be given in terms of the adaptive robust Bellman equation associated to it.

Towards this end we will need to solve for functions $W_t, \ t\in \cT,$ the following adaptive robust Bellman equations (recall that $y=(x,c)$)
\begin{align}
  W_T(y) & = \ell(x), \quad y\in E_Y,
    \nonumber \\
W_t(y) & = \inf_{a\in A} \sup_{ \theta \in \tau(t,c)}
\int_{E_Y} W_{t+1}(y')
   Q( dy'\mid t,y,a,\theta ) , \quad
     y\in E_Y, \   t=T-1, \ldots, 0, \label{eq:bellmanEquationrobustIII}
\end{align}
and to compute the related optimal selectors $\varphi^*_t, \ t\in \cT'$.

In Lemma~\ref{lemma:WTusc} below, under some additional technical assumptions,  we will show that the optimal selectors in \eqref{eq:bellmanEquationrobustIII} exist; namely, for any $t\in \cT'$,  and any $y=(x,c)\in E_Y$, there exists a measurable mapping $\varphi^*_t\, : \, E_Y \rightarrow A,$ such that
\[
 W_t(y)  =  \sup_{ \theta \in \tau(t,c)}\int_{E_Y} W_{t+1}(y')
   Q( dy'\mid t,y,\varphi^*_t(y),\theta).
\]
In order to proceed, to simplify the argument, we will assume that under measure $\bP_\theta$, for each $t\in \cT$, the random variable $Z_t$ has a density   with respect to the Lebesgue measure, say $f_Z(z;\theta), \ z\in\bR^m$. We will also assume that the set $A$ of available actions is finite. In this case, the problem  \eqref{eq:bellmanEquationrobustIII} becomes
\begin{align}
  W_T(y) & = \ell(x), \quad y\in E_Y, \nonumber \\
W_t(y) & = \min_{a\in A} \sup_{ \theta \in \tau(t,c)}
\int_{\bR^m} W_{t+1}(\mathbf{T}(t,y, a, z))f_Z(z; \theta)dz ,
   \quad y\in E_Y, \   t=T-1, \ldots, 0.
    \label{eq:BellmanEqRobust4}
\end{align}
where $T(t,y, a, z)$ is given in \eqref{eq:T}.
Additionally, we take the standing assumptions that
\begin{enumerate}[(i)]
\item for any $a$ and $z$, the function $S(\cdot,a,z)$ is
continuous.
  \item For each $z$, the function $f_Z(z;\cdot)$ is continuous in $\theta$.
  \item $\ell$ is continuous and bounded.
  \item For each $t\in \cT'$, the function $R(t,\cdot,\cdot)$ is continuous.
\end{enumerate}
Then, we have the following result.

\begin{lemma}\label{lemma:WTusc} The functions $W_t,\ t=T,T-1,\ldots,0,$ are upper semi-continuous (u.s.c.), and the optimal selectors $\varphi^*_t ,\ t=T,T-1,\ldots,0,$ in \eqref{eq:bellmanEquationrobustIII} exist.
\end{lemma}
\begin{proof}
The function $W_T$ is continuous.
Since $\mathbf{T}(T-1,\cdot,\cdot,z)$ is continuous, then, $W_T(\mathbf{T}(T-1,\cdot,\cdot,z))$ is continuous.
Consequently, 
the function
\[
w_{T-1}(y,a,\theta)=\int_{\bR} W_{T}(\textbf{T}(T-1,y, a, z)) f_Z(z; \theta)dz
\]
is continuous, and thus u.s.c.

Next we will apply \cite[Proposition 7.33]{BertsekasShreve1978Book} by taking (in the notations of \cite{BertsekasShreve1978Book})
\begin{align*}
\mathrm{X} & = E_Y\times A = \bR^n\times\bR^d\times A,  \quad \mathrm{x}=(y,a), \\
\mathrm{Y} & =  \boldsymbol{\Theta}, \quad  \mathrm{y}=\theta, \\
\mathrm{D} & =\bigcup_{(y,a)\in E_Y \times A} \set{(y,a)} \times  \tau (T-1,c), \\
f(\mathrm{x},\mathrm{y}) & = - w_{T-1}(y,a,\theta).
\end{align*}
Recall that in view of the prior assumptions, $\mathrm{Y}$ is compact.  Clearly $\mathrm{X}$ is metrizable.
From the above, $f$ is lower semi-continuous (l.s.c). Also note that the cross section $\mathrm{D}_\mathrm{x} = {\mathrm D}_{(y,a)} = \set{\theta \in \boldsymbol{\Theta}\, :\, (y,a,\theta) \in \mathrm{D}}$ is given by $\mathrm{D}_{(y,a)}(t)=\tau (t,c)$.
Hence, by \cite[Proposition 7.33]{BertsekasShreve1978Book}, the function
\[
\widetilde w_{T-1}(y,a)= \inf_{\theta \in \tau(T-1,c)} (-w_{T-1}(y,a,\theta)) ,\quad (y,a)\in E_Y \times A
\]
is l.s.c.. Consequently, the function $W_{T-1}= {\inf _{a\in A}}(-\widetilde w_{T-1}(y,a))$ is u.s.c., and there exists an optimal selector $\varphi^*_{T-1}$.
The rest of the proof follows in the analogous way.
\end{proof}

The following proposition is the key result in this section.
\begin{proposition}\label{prop:main} For any $h_t\in H_t$, and $t\in\cT$, we have
\begin{equation}\label{opt}
U^*_t(h_t)=W_t(y_t).
\end{equation}
Moreover, the policy $\varphi^*$ constructed from the selectors in
\eqref{eq:bellmanEquationrobustIII} is robust-optimal, that is
\begin{equation}\label{opt1}
U^*_t(h_t)=U_t(\varphi_t ^*,h_t), \quad t\in \cT'.
\end{equation}
\end{proposition}

\begin{proof}
We proceed similarly as in the proof of \cite[Theorem 2.1]{Iyengar2005}, and via backward induction in $t=T,T-1,\ldots,1,0$.

Take $t=T$. Clearly, $U^*_T(h_T)=W_T(y_T)$. For $t=T-1$ we have
\begin{align*}
U^*_{T-1}(h_{T-1})& = \inf_{\varphi^{T-1}=\varphi_{T-1}\in {\cA}^{T-1}}\sup_{\mathbb{Q}\in {\mathcal Q}^{\varphi^{T-1}, {\boldsymbol \Psi}^{T-1}}_{h_{T-1}}}
\bE^{\mathbb{Q}}\ell(X_T)\\
 & =\inf_{\varphi^{T-1}=\varphi_{T-1}\in {\mathcal A}^{T-1}}\sup_{ \theta \in \tau(T-1,c_{T-1})}
 \int_{E_Y}  U^*_{T}(h_{T-1},y) Q( dy\mid T-1,y_{T-1},\varphi_{T-1}(h_{T-1}),\theta )\\
 & = \inf_{\varphi^{T-1}=\varphi_{T-1}\in {\mathcal A}^{T-1}}\sup_{ \theta \in \tau(T-1,c_{T-1})} \int_{E_Y} W_{T}(y) \,Q( dy\mid T-1,y_{T-1},\varphi_{T-1}(h_{T-1}),\theta )\\
 &= \inf_{a\in {A}}\sup_{ \theta \in \tau(T-1,c_{T-1})}\int_{E_Y} W_{T}(y)  \, Q( dy\mid T-1,y_{T-1},a,\theta )=W_{T-1}(y_{T-1}).
\end{align*}

For $t=T-1,\ldots,1,0$ we have by induction
\begin{eqnarray*} U^*_t(h_t)&=&\inf_{\varphi^t\in {\mathcal
A}^t}\sup_{\mathbb{Q}\in {\mathcal Q}^{\varphi^t, {\boldsymbol \Psi}^t}_{h_t}} \bE^{\mathbb{Q}}\ell (X_T)\\
&=&\inf_{\varphi^t=(\varphi_t,\varphi^{t+1})\in {\mathcal A}^t}\sup_{ \theta \in \tau(t,c_t)}
 \int_{E_Y}  \sup_{\widehat {\mathbb{Q}}\in {\mathcal Q}^{\varphi^{t+1}, {\boldsymbol \Psi}^{t+1}}_{h_t,y}}
 \bE^{\widehat {\mathbb{Q}}}\ell (X_T)
Q( dy\mid t,y_t,\varphi_t(h_t),\theta )\\&\geq &\inf_{\varphi^t=(\varphi_t,\varphi^{t+1})\in {\mathcal A}^t}\sup_{ \theta \in \tau(c_t,t)} \int_{E_Y}
U^*_{t+1}(h_t,y) \,Q( dy\mid t,y_t,\varphi_t(h_t),\theta )\\&=
&\inf_{a\in {A}}\sup_{ \theta \in \tau(t,c_t)}\int_{E_Y}
U^*_{t+1}(h_t,y)
  \, Q( dy\mid t,y_t,a,\theta )\\&=
&\inf_{a\in {A}}\sup_{ \theta \in \tau(t,c_t)}\int_{E_Y}
W_{t+1}(y)
  \, Q( dy\mid t,y_t,a,\theta )=W_{t}(y_t).
\end{eqnarray*}
Now, fix $\epsilon >0$, and let $\varphi^{t+1,\epsilon}$ denote an
$\epsilon$-optimal control process starting at time $t+1$, so that
\[
U_{t+1}(\varphi^{t+1,\epsilon},h_{t+1})\leq U^*_{t+1}(h_{t+1})+\epsilon.
\]
Then we have
\begin{align*}
U^*_t(h_t)& =\inf_{\varphi^t\in {\mathcal A}^t}\sup_{\mathbb{Q}\in {\mathcal Q}^{\varphi^t, {\boldsymbol \Psi}^t}_{h_t}}
    \bE^{\mathbb{Q}}\ell (X_T)\\
 & =\inf_{\varphi^t=(\varphi_t,\varphi^{t+1})\in {\mathcal A}^t}\sup_{ \theta \in \tau(t,c_t)}\int_{E_Y}
    \sup_{\widehat {\mathbb{Q}}\in {\mathcal Q}^{\varphi^{t+1}, {\boldsymbol \Psi}^{t+1}}_{h_t,y}}\bE^{\widehat {\mathbb{Q}}}\ell (X_T)
    Q( dy\mid t,y_t,\varphi_t(h_t),\theta )\\
&\leq \inf_{\varphi^t=(\varphi_t,\varphi^{t+1})\in {\mathcal A}^t}\sup_{ \theta \in \tau(t,c_t)} \int_{E_Y}
    \sup_{\widehat {\mathbb{Q}}\in {\mathcal Q}^{\varphi^{t+1,\epsilon}, {\boldsymbol \Psi}^{t+1}}_{h_t,y}}
    \bE^{\widehat {\mathbb{Q}}}\ell (X_T)Q( dy\mid t,y_t,\varphi_t(h_t),\theta )\\
& \leq  \inf_{a\in {A}}\sup_{ \theta \in \tau(t,c_t)}\int_{E_Y}  U^*_{t+1}(h_t,y)\,
    Q(dy\mid t,y_t,a;\theta )+\epsilon  \\
&=  \inf_{a\in {A}}\sup_{ \theta \in \tau(t,c_t)}\int_{E_Y}  W_{t+1}(y)\,
    Q(dy\mid t,y_t,a;\theta )+\epsilon=W_{t}(y_t)+\epsilon.
\end{align*}
Since $\epsilon$ was arbitrary, the proof of \eqref{opt} is done.
Equality \eqref{opt1} now follows easily.
\end{proof}

\section{Example: Dynamic Optimal Portfolio Selection}\label{sec:examples}
In this section we will present an example that illustrates the adaptive robust control methodology.

We follow here the set up of \cite{BrandtEtAl2005} in the formulation of our dynamic optimal portfolio selection.
As such, we consider the classical dynamic optimal asset allocation problem, or dynamic optimal portfolio selection, when an investor is deciding at time $t$ on investing in a risky asset and a risk-free banking account by maximizing the expected utility $u(V_T)$ of the terminal wealth, with $u$ being a given utility function. The underlying market model is subject to the type of uncertainty that has been introduced in Section~\ref{sec:ocsmu}.

Denote by $r$ the constant risk-free interest rate and by $Z_{t+1}$ the excess return on the risky asset from time $t$ to $t+1$. We assume that the process $Z$ is observed. The dynamics of the wealth process produced by a self-financing trading strategy is given by
\begin{equation}\label{eq:wealthEx1}
  V_{t+1} = V_t(1+r + \varphi_t Z_{t+1}),\quad t\in {\cal T}',
\end{equation}
with the initial wealth $V_0=v_0$, and where $\varphi_t$ denotes the proportion of the portfolio wealth invested in the risky asset from time $t$ to $t+1$. We assume that the process $\varphi$ takes finitely many values, say $a_i,\ i=1,\ldots,N$ where $a_i\in [0,1].$ Using the notations from Section~\ref{sec:robust}, here we have that $X_t=V_t$, and setting $x=v$ we get
$$
S(v,a,z)=v(1+r+az), \quad \ell(v)=-u(v), \quad A=\set{a_i,\ i=1,\ldots,N}.
$$
We further assume that the excess return process $Z_{t}$ is an  i.i.d. sequence of Gaussian random variables with mean $\mu$ and variance $\sigma^2$. Namely, we put
$$
Z_{t} = \mu +\sigma \varepsilon_{t},
$$
where $\varepsilon_t, \ t\in \cT'$ are i.i.d. standard Gaussian random variables.
The model uncertainty comes from the unknown parameters $\mu$ and/or $\sigma$.
We will discuss two cases: Case 1 - unknown mean $\mu$ and known standard deviation $\sigma$, and Case II - both $\mu$ and $\sigma$ are unknown.

\smallskip\noindent
\textit{Case I.} Assume that $\sigma$ is known, and the model uncertainty comes only from the unknown parameter $\mu^*$.
Thus, using the notations from Section~\ref{sec:robust}, we have that $\theta^* = \mu^*,$ $\theta = \mu,$ and we take $C_t=\widehat \mu_t$, $\boldsymbol\Theta =[\underline{\mu},\overline{\mu}]\subset\bR$, where $\widehat{\mu}$ is  an estimator of $\mu$, given the observations $Z$, that takes values in $\boldsymbol{\Theta}$. For the detailed discussion on the construction of such estimators we refer to \cite{BCC2016}.
For this example, it is enough to take $\widehat{\mu}$ the Maximum Likelihood Estimator (MLE), which is the sample mean in this case, projected appropriately on $\boldsymbol\Theta$. Formally, the recursion construction of $\widehat{\mu}$ is defined as follows:
\begin{equation}\label{eq:esti_proj}
\begin{aligned}
    \widetilde{\mu}_{t+1}&=\frac{t}{t+1}\widehat{\mu}_{t}+{\frac{1}{t+1}} Z_{t+1},\\
    \widehat{\mu}_{t+1}&=\mathbf{P}(\widetilde{\mu}_{t+1}), \qquad t\in\cT',
\end{aligned}
\end{equation}
with $\widehat{\mu}_0=c_0$, and where $\mathbf{P}$ is the projection to the closest point in $\boldsymbol\Theta$, i.e.
$\mathbf P(\mu) = \mu$ if $\mu\in[\underline{\mu},\overline{\mu}]$, $\mathbf P(\mu) = \underline{\mu}$ if $\mu<\underline{\mu}$, and $\mathbf P(\mu) = \overline{\mu}$ if $\mu> \overline{\mu}$. We take as the initial guess $c_0$ any point in $\boldsymbol\Theta$. It is immediate to verify that
\[
 R(t, c,z)=\mathbf{P}\left(\frac{t}{t+1} c +\frac{1}{t+1} z\right)
\]
is continuous in $c$ and $z$. Putting the above together we get that the function $\mathbf{T}$ corresponding to \eqref{eq:T} is given by
$$
\mathbf{T}(t,v,c,a,z) = \left(v(1+r+az), \frac{t}{t+1}c+\frac{1}{t+1}z\right).
$$
Now, we note that the $(1-\alpha)$-confidence region  for $\mu^*$ at time $t$ is given as\footnote{We take this interval to be closed, as we want it to be compact.}
\[
 \mathbf{\Theta}_t=\tau(t,\widehat \mu_t),
\]
 where
 \[
 \tau(t,c)=\left [c-\frac{\sigma}{\sqrt t}q_{\alpha/2}, c +\frac{\sigma}{\sqrt t}q_{\alpha/2}\right ],
 \]
and  where  $q_\alpha$ denotes the $\alpha$-quantile  of a standard normal distribution. With these at hand we define the kernel $Q$ according to \eqref{eq:QB-bar}, and the set of probability measures $\cQ^{\varphi,\boldsymbol\Psi}_{h_0}$ on canonical space by \eqref{eq:prob}.

Formally, the investor's problem\footnote{The original investor's problem is   $\sup_{\varphi \in\cA} \inf_{\bQ\in\cQ_{h_0}^{\varphi,\boldsymbol\Psi}} \bE_\bQ[u(V_T)]$, where $\cA$ is the set of self-financing trading strategies}, in our notations and setup, is formulated as follows
\begin{equation}\label{eq:ex1Problem}
  \inf_{\varphi \in\cA} \sup_{\bQ\in\cQ_{h_0}^{\varphi,\boldsymbol\Psi}} \bE_\bQ[-u(V_T)]
\end{equation}
where $\cA$ is the set of self-financing trading strategies.

The corresponding  adaptive robust Bellman equation becomes
\begin{equation}\label{eq:BellmanEx1}
\begin{cases}
  W_T(v, c) = -u(v), \\
  W_t(v, c) = \inf_{a\in A} \sup_{\mu\in\tau_\alpha(t, c)}
  \bE\left[ W_{t+1}\left(T(t,v, c ,a, \mu + \sigma \varepsilon_{t+1})\right)\right],
\end{cases}
\end{equation}
where the expectation $\bE$ is with respect to a standard one dimensional Gaussian distribution. In view of  Proposition \ref{prop:main}, the function $W_t(v,c)$ satisfies
$$
W_t(v, c) = \inf_{\varphi^t\in {\mathcal A}^t}\sup_{\mathbb{Q}\in {\mathcal Q}^{\varphi^t, {\boldsymbol \Psi}^t}_{h'_{t}} }
\bE_{\mathbb{Q}}[-u(V_T)],
$$
where $h'_t=(h_{t-1}, v, c)$ for any $h_{t-1}=(v_0, c_0,\ldots,v_{t-1}, c_{t-1})\in \mathbf{H}_{t-1}$.
So, in particular,
$$
W_0(v_0, c_0)=\inf_{\varphi \in\cA} \sup_{\bQ\in\cQ_{h_0}^{\varphi,\boldsymbol\Psi}} \bE_\bQ[-u(V_T)],
$$
and the optimal selector $\varphi^*_t,t\in\cT$, from \eqref{eq:BellmanEx1} solves the original investor's allocation problem.

To further reduce the computational complexity, we consider a CRRA utility of the form  $u(x) = \frac{x^{1-\gamma}}{1-\gamma}$, for $x\in \bR$,  and some $\gamma\neq1$. In this case, we have
$$
u(V_T) = u(V_t)\left[\prod_{s=t}^{T-1} (1+r+\varphi_s Z_{s+1})\right]^{1-\gamma}, \quad t\in \cT'.
$$
Note that
$$
W_t(v,c) = -u(v)\cdot
\begin{cases}
\sup_{\varphi^t\in\cA^t} \inf_{\mathbb{Q}\in {\mathcal Q}^{\varphi^t, {\boldsymbol \Psi}^t}_{h'_{t}} } \bE_{\mathbb{Q}}\left[ \left(\prod_{s=t}^{T-1}(1+r+\varphi_s Z_{s+1})\right)^{1-\gamma}\right],  & \gamma <1, \\
\inf_{\varphi^t\in\cA^t} \sup_{\mathbb{Q}\in {\mathcal Q}^{\varphi^t, {\boldsymbol \Psi}^t}_{h'_{t}} } \bE_{\mathbb{Q}}\left[ \left(\prod_{s=t}^{T-1}(1+r+\varphi_s Z_{s+1})\right)^{1-\gamma}\right], & \gamma>1.  \end{cases}
$$
Next, following the ideas presented in \cite{BrandtEtAl2005}, we will  prove that for $t\in \cT$ and any $c\in[a,b]$ the ratio $W_t(v,c)/v^{1-\gamma}$ does not depend on $v$, and that the functions $\widetilde{W}_t$ defined as $\widetilde{W}_t(c)=W_t(v,c)/v^{1-\gamma}$ satisfy the following backward recursion
\begin{equation}\label{eq:BellmanEx2-1}
\begin{cases}
 \widetilde W_T(c) = \frac{1}{1-\gamma}, \\
 \widetilde W_t(c) = \inf_{a\in{A}} \sup_{\mu\in\tau_\alpha(t,c)}
 \bE\left[(1+r+a (\mu+\sigma \varepsilon_{t+1}))^{1-\gamma}\widetilde W_{t+1}(\frac{t}{t+1}c +
 \frac{1}{t+1} ( \mu+\sigma \varepsilon_{t+1}) ) \right], \ t\in \cT'.
\end{cases}
\end{equation}
We will show this by backward induction in $t$. First, the equality $\widetilde W_T(c) = \frac{1}{1-\gamma}$ is obvious. Next, we fix  $t\in \cT'$, and we assume that $W_{t+1}(v,c)/v^{1-\gamma}$ does not depend on $v$. Thus, using \eqref{eq:BellmanEx1} we obtain
\[
\frac{W_{t}(v,c)}{v^{1-\gamma}}= \inf_{a\in{ A}} \sup_{\mu\in\tau_\alpha(t, c)}
\bE\left[ (1+r+a(\mu +\sigma \varepsilon_{t+1}))^{1-\gamma}\widetilde W_{t+1}\left(T(t,v,c,a,\mu+\sigma \varepsilon_{t+1})\right)\right]
\]
does not depend on $v$ since $\widetilde W_{t+1}$ does not depend on its first argument.

We finish this example with several remarks regarding the numerical implementations aspects of this problem.
By considering CRRA utility functions, and with the help of the above substitution, we reduced the original recursion problem \eqref{eq:BellmanEx1} to \eqref{eq:BellmanEx2-1} which has a lower dimension, and which consequently significantly reduces the  computational complexity of the problem.

In the next section we compare  the strategies (and the corresponding wealth process) obtained by the adaptive robust method,  strong robust methods, adaptive control method, as well as by considering the case of no model uncertainty when the true model is known.

Assuming that the true model is known, the trading strategies are computed by simply solving the optimization problem \eqref{eq:gen-1}, with its corresponding Bellman equation
\begin{equation}\label{eq:Ex1BellmanTrueModel}
\begin{cases}
  \widetilde W_T = \frac{1}{1-\gamma}, \\
 \widetilde W_t = \inf_{a\in{A}}  \bE\left[(1+r+a ( \mu^{*}+\sigma \varepsilon_{t+1}))^{1-\gamma}\widetilde W_{t+1})
   \right], \ t\in \cT'.
\end{cases}
\end{equation}
Similar to the derivation of \eqref{eq:BellmanEx2-1}, one can show that the Bellman equation for the robust control problem takes the form
\begin{equation}\label{eq:Ex1BellmanStrongRobust}
\begin{cases}
  \widetilde W_T = \frac{1}{1-\gamma}, \\
 \widetilde W_t = \inf_{a\in{A}} \sup_{ \mu\in\boldsymbol\Theta}
  \bE\left[(1+r+a ( \mu+\sigma \varepsilon_{t+1}))^{1-\gamma}\widetilde W_{t+1}
   \right], \ t\in \cT'.
\end{cases}
\end{equation}
Note that the Bellman equations \eqref{eq:Ex1BellmanTrueModel} and  \eqref{eq:Ex1BellmanStrongRobust} are recursive scalar sequences that can be computed numerically efficiently  with no state space discretization required.
The adaptive control strategies are obtained by solving, at each time iteration $t$, a Bellman equation similar to \eqref{eq:Ex1BellmanTrueModel}, but by iterating backward up to time $t$ and where $\mu^{*}$ is replaced by its estimated value $\widehat{\mu}_{t}$.

To solve the corresponding adaptive control problem, we first perform the optimization phase. Namely, we solve the Bellman equations \eqref{eq:Ex1BellmanTrueModel} with $\mu^*$ replaced by $\mu$, and for all $\mu\in\boldsymbol\Theta$. The optimal selector is denoted by $\varphi_t^\mu, \ t\in\cT'$. Next, we do the adaptation phase. For every $t\in \{0,1,2,\ldots,T-1\},$ we compute the pointwise estimate $\widehat\mu_t$ of $\theta ^*$, and apply the certainty equivalent control  $\varphi_t=\varphi^{\widehat \mu_t}_t$. For more details see, for instance, \cite{KumarVaraiya2015Book}, \cite{ChenGuo1991-Book}.

\smallskip\noindent
\textit{Case II.} Assume that both $\mu$ and $\sigma$ are the unknown parameters, and thus in the notations of Section~\ref{sec:robust}, we have $\theta^*=(\mu^*,(\sigma^*)^2)$, $\theta=(\mu,\sigma^2)$, $\boldsymbol\Theta= [\underline{\mu},\overline{\mu}]\times[\underline{\sigma}^2,\overline{\sigma}^2]\subset\bR\times\bR_+$, for some fixed $\underline{\mu},\overline{\mu}\in\bR$ and $\underline{\sigma}^2, \overline{\sigma}^2\in\bR_+$.
Similar to the Case~I, we take the MLEs for $\mu^*$ and $(\sigma^*)^2$, namely the sample mean and respectively the sample variance, projected appropriately to the rectangle $\boldsymbol\Theta$.  It is shown in \cite{BCC2016}  that the following recursions hold true
\begin{align*}
   \widetilde{\mu}_{t+1} &=\frac{t}{t+1}\widehat{\mu}_{t}+\frac{1}{t+1} Z_{t+1}, \\
    \widetilde{\sigma}_{t+1}^2 &=\frac{t}{t+1}\widehat{\sigma}^2_{t}+\frac{t}{(t+1)^{2}}(\widehat{\mu}_{t}- Z_{t+1})^2,\\
    (\widehat{\mu}_{t+1},\widehat{\sigma}_{t+1}^2)&=\mathbf{P}(\widetilde{\mu}_{t+1},\widetilde{\sigma}_{t+1}^2), \quad t=1,\ldots,T-1,
  \end{align*}
with some initial guess $\widehat{\mu}_0=c_0'$, and $\widehat{\sigma}_0^2=c_0''$, and where $\mathbf{P}$ is the projection\footnote{We refer to \cite{BCC2016} for precise definition of the projection $\boldsymbol P$, but essentially it is defined as the closest point in the set $\boldsymbol\Theta$.} defined similarly as in \eqref{eq:esti_proj}.
Consequently, we put $C_t=(C_t',C_t'')=(\widehat\mu_t, \widehat\sigma^2_t), t\in\cT$, and respectively we have
\[
R(t,c,z)=\mathbf{P}\left(\frac{t}{t+1}c'+\frac{1}{t+1}z,\frac{t}{t+1}c''+\frac{t}{(t+1)^{2}}(c'-z)^2\right),
\]
with $c=(c',c'')$. Thus, in this case, we take
$$
\mathbf{T}(t,v,c,a,z) = \left(v(1+r+az), \frac{t}{t+1}c'+\frac{1}{t+1}z, \frac{t}{t+1}c''+\frac{t}{(t+1)^{2}}(c'-z)^2 \right).
$$
It is also shown in \cite{BCC2016} that here the $(1-\alpha)$-confidence region  for $(\mu^*,(\sigma^*)^2)$ at time $t$ is an ellipsoid given by
\begin{align*}
\mathbf{\Theta}_t=\tau(t,\widehat \mu_t,\widehat \sigma^2_t), \quad
\tau(t,c)=\left\{ c=(c',c'')\in\bR^2 \ : \ \frac{t}{c''}(c'-\mu)^2 + \frac{t}{2 (c'')^2}(c'' -\sigma^{2})^2\leq \kappa\right\},
\end{align*}
where $\kappa$ is the $(1-\alpha)$ quantile of the $\chi^2$ distribution with two degrees of freedom.
Given all the above, the adaptive robust Bellman equations are derived by analogy to \eqref{eq:BellmanEx1}-\eqref{eq:BellmanEx2-1}.
Namely, $\widetilde W_T(c) = \frac{1}{1-\gamma}$ and, for any $t\in \cT'$,
\begin{align}\label{eq:Bellman2D}
\widetilde W_t(c) =  \sup_{a\in A} & \inf_{(\mu, \sigma^{2})\in\tau(t,c)}
  \bE\Big[(1+r+a ( \mu+\sigma \varepsilon_{t+1}))^{1-\gamma} \\
&  \times \widetilde W_{t+1}\left(\frac{t}{t+1}{c'} + \frac{1}{t+1} ( \mu+ \sigma \varepsilon_{t+1}), \frac{t}{t+1}c''+\frac{t}{(t+1)^{2}}(c'-(\mu+\sigma \varepsilon_{t+1}))^2\right) \Big].  \nonumber
\end{align}
The Bellman equations for the true model, and strong robust method are computed  similarly to \eqref{eq:Ex1BellmanTrueModel} and  \eqref{eq:Ex1BellmanStrongRobust}.

\subsection{Numerical Studies}\label{sec:numerical_studies}
In this section,  we compute the terminal wealth generated by the optimal adaptive robust controls for both cases that are discussed in Section~\ref{sec:examples}.  In addition, for these two cases, we compute the terminal wealth by assuming that the true parameters $\mu^*$ and $(\sigma^*)^2$  are known, and then using respective optimal controls; we call this the true model control specification. We also compute the terminal wealth obtained from using the optimal adaptive controls. Finally,  we compute the terminal wealth using optimal robust controls, which, as it turns out, are the same in Case I as the optimal  strong robust controls.  We perform an analysis of the results and we compare the four methods that are considered.

 In the process, we first numerically solve the respective Bellman equation for each of the four considered methods. This is done by backward induction, as usual.

Note that expectation operator showing in the Bellman equation is just an expectation over the standard normal distribution. For all considered control methods and for all simulation conducted, we approximate the standard normal distribution by a $10$-points optimal quantizer (see, e.g., \cite{PagesPrintems2003}).

  The Bellman equations for the true model control specification and for the adaptive control are essentially the same, that is \eqref{eq:Ex1BellmanTrueModel}, except that, as already said above, in the adaptive control implementation, we use the certainty equivalent approach: at time $t$ the Bellman equation is solved for the time-$t$ point estimate of the parameter $\mu^*$ (in Case I), or for the time-$t$ point estimates of the parameters $\mu^*$ and $(\sigma^*)^2$ (in Case II).

In the classic robust control application the Bellman equation \eqref{eq:Ex1BellmanStrongRobust} is solved as is normally done in the dynamic min-max game problem, in both Case I and Case II.

In the adaptive robust control application the Bellman equation \eqref{eq:BellmanEx2-1} is a recursion on a real-valued cost-to-go function $\widetilde W$. In Case I, this recursive equation is numerically solved  by discretizing the state space associated with state variable $\widehat  \mu_{t}$. In our illustrative example, this discretization  of the state space has been done  by simulating sample paths of the state process $\widehat  \mu_{\cdot}$ up to horizon time $T$.  This simulation has been made under the true model\footnote{Of course, this cannot be done if the control method is applied on genuinely observed market data since the market model is (by nature) not known. One possible solution would be first to estimate the model parameters based on a past history of $Z$ and second to generate sample paths of the state process $\widehat  \mu_{\cdot})$ according to this estimated model.}. At each time $t$, the state space grid has been defined as the collection of sample values taken by this process at time $t$. Analogous procedure has been applied in Case II with regard to state variables $\widehat  \mu_{t}$ and  $\widehat \sigma^2_{t}$.

\noindent\textbf{Case I.} In order to implement adaptive robust control method for solving the optimal allocation problem, we start by constructing a grid in both time and space. The grid consists of a number of simulated paths of $\hat{\mu}$. Then, we solve equations~\eqref{eq:BellmanEx2-1} at all grid points for the optimal trading strategies.

As stated above, the implementation  of the adaptive control method includes two steps. First, for each $\mu$ in the uncertain set $\boldsymbol \Theta=[\underline{\mu},\overline{\mu}]$, we solve the following Bellman equations:
\begin{equation*}
\begin{cases}
  \widetilde W_T = -\frac{1}{1-\gamma}, \\
 \widetilde W_t = \inf_{a\in{A}}\bE\left[\widetilde W_{t+1}(1+r+a (\mu+\sigma \varepsilon_{t+1}))^{1-\gamma} \right], \ t\in \cT'.
\end{cases}
\end{equation*}
It is clear that at each time $t$, the optimal control is parameterized by  $\mu$.
With this in mind, at each time $t\in\cT'$, we choose the control value with $\hat{\mu}_t$ replacing $\mu$ in the formula for the optimal control.

Under classical robust control method, the investor's problem becomes
\begin{equation}\label{eq:num-robust}
  \inf_{\varphi\in\cA}\sup_{\theta\in\boldsymbol \Theta}\bE_{\theta}[-u(V_T)]=-v_0^{1-\gamma}\sup_{\varphi\in\cA}\inf_{\mu\in[\underline{\mu},\overline{\mu}]}\bE\left[\frac{1}{1-\gamma}\left(\prod_{t=0}^{T-1}(1+r+\varphi_t(\mu+\sigma\varepsilon_{t+1}))\right)^{1-\gamma}\right].
\end{equation}
The inner infimum problem in \eqref{eq:num-robust} is attained  at $\mu=\underline{\mu}$ as long as $1+r+\varphi_t(\mu+\varepsilon_{t+1})\geq0$ for each $t\in\cT'$.
Such condition will be satisfied for rather reasonable choice of the risk free rate $r$ and quantization of $\varepsilon_{t+1}$.
Accordingly, the robust control problem becomes
\begin{eqnarray*}
  \inf_{\varphi\in\cA}\sup_{\theta\in\boldsymbol \Theta}\bE_{\theta}[-u(V_T)]=\frac{v_0^{1-\gamma}}{1-\gamma}\inf_{\varphi\in\cA}\bE\left[-\left(\prod_{t=0}^{T-1}(1+r+\varphi_t(\underline{\mu}+\sigma\varepsilon_{t+1}))\right)^{1-\gamma}\right].
\end{eqnarray*}
The corresponding Bellman equation becomes
\begin{equation}\label{eq:rb-bellman}
\begin{cases}
  \widetilde W_T = -\frac{1}{1-\gamma}, \\
 \widetilde W_t = \inf_{a\in{A}}\bE\left[\widetilde W_{t+1}(1+r+a (\underline\mu+\sigma \varepsilon_{t+1}))^{1-\gamma} \right], \ t\in \cT'.
\end{cases}
\end{equation}
We compute the robust optimal strategy by solving equation~\eqref{eq:rb-bellman} backwards.

It can be shown that for this allocation problem, the strong robust control  problem~\eqref{prob1R-strong} is also solved via the Bellman equation~\eqref{eq:rb-bellman}. Hence, in this case, strong robust control method and robust control method provide the same result.


For numerical study we choose the parameter set as $\boldsymbol\Theta=[-1,1]$, and we consider a set of time horizons $T=0.1, 0.2, \ldots, 0.9, 1$.
The other parameters are chosen as follows
\begin{eqnarray*}
    V_0=100,\quad r=0.02,\quad \alpha=0.1 ,\quad \gamma=5,\quad \sigma=0.3,\quad \mu^*=0.07, \quad \widehat \mu_0=0.1.
\end{eqnarray*}
For every $T$, we compute the terminal wealth $V_T$ generated by application of the optimal strategies corresponding to four control methods mentioned above: adaptive robust, classical robust, adaptive and the optimal control (assuming the true parameters are known in the latter case).  In each method we use 1000 simulated paths of the risky asset and $300T$ rebalancing time steps.  Finally, we use the acceptability index Gain-to-Loss Ratio (GLR)
$$
\textrm{GLR}(V)=
\begin{cases}
  \frac{\bE_{\theta^*}[e^{-rT}V_T-V_0]}{\bE_{\theta^*}[(e^{-rT}V_T-V_0)^-]}, \quad & \bE_{\theta^*}[e^{-rT}V_T-V_0]>0,\\
 0, \quad &\text{otherwise,}
\end{cases}
$$
and 95\% Value-at-Risk, $\var(V_T)=\inf\{v\in\bR: \bP_{\theta^*}(V_T+v<0)\leq95\%\}$,
to compare the performance of every method.

\begin{figure}[!ht]
    \centering
    \begin{subfigure}[t]{0.48\textwidth}
        \centering
        \includegraphics[width=\linewidth]{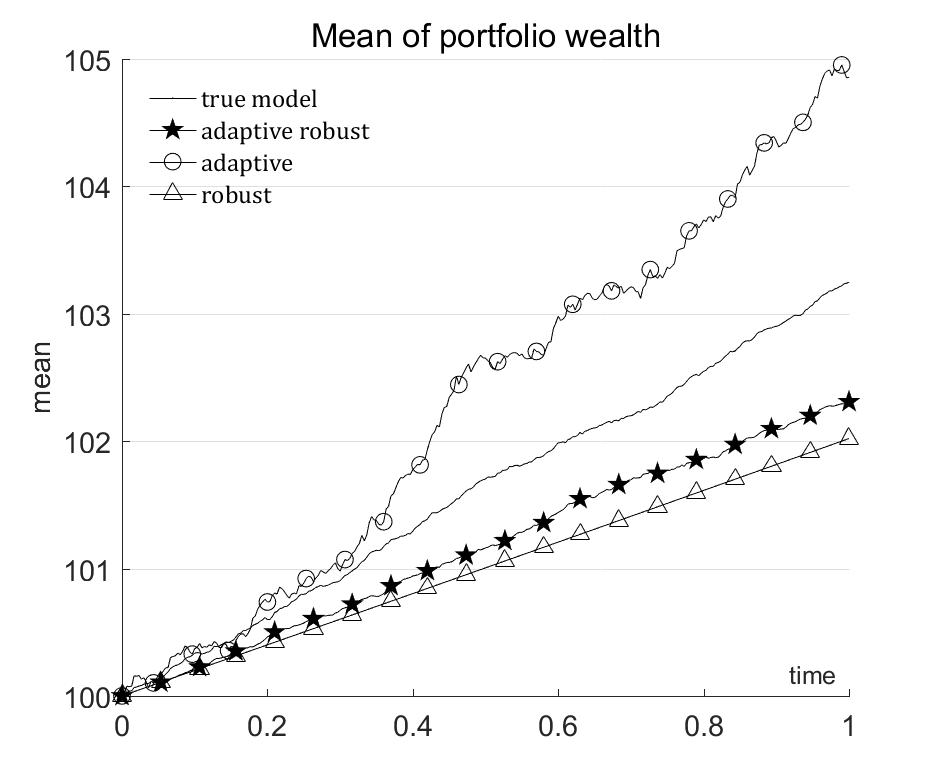}
    \end{subfigure}
    \hfill
    \begin{subfigure}[t]{0.48\textwidth}
        \centering
        \includegraphics[width=\linewidth]{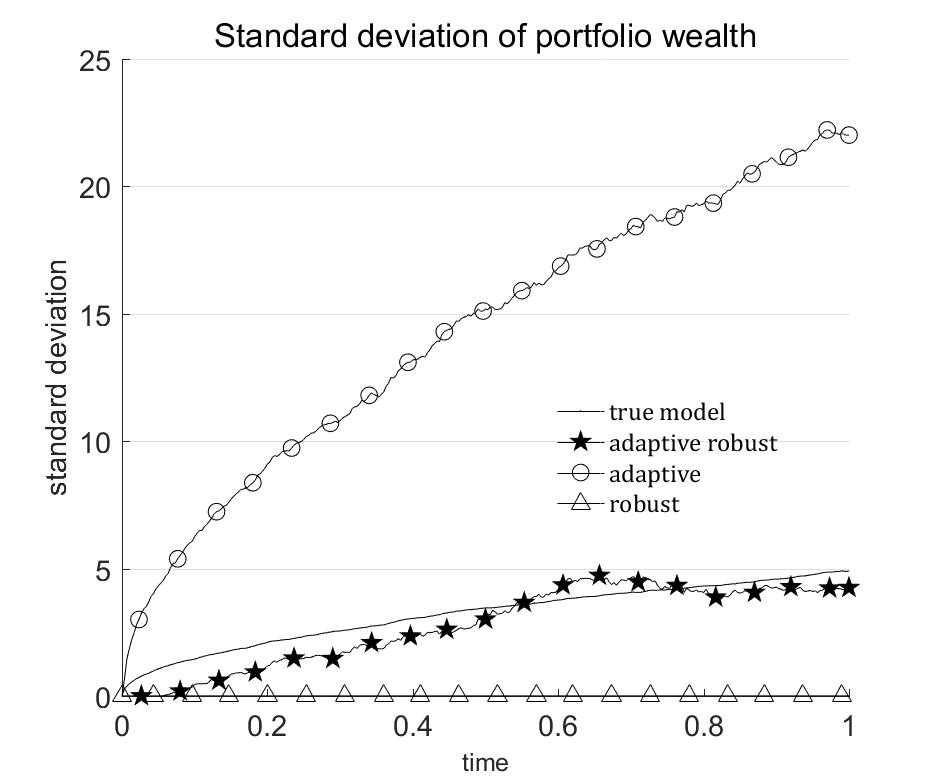}
    \end{subfigure}
    \caption{Time-series of portfolio wealth means, and standard deviations. Unknown mean.}
    \label{fig:PortValueMeanAndStDev-1d}
\end{figure}

\begin{figure}[!ht]
       \centering
    \begin{subfigure}[t]{0.48\textwidth}
        \centering
        \includegraphics[width=\linewidth]{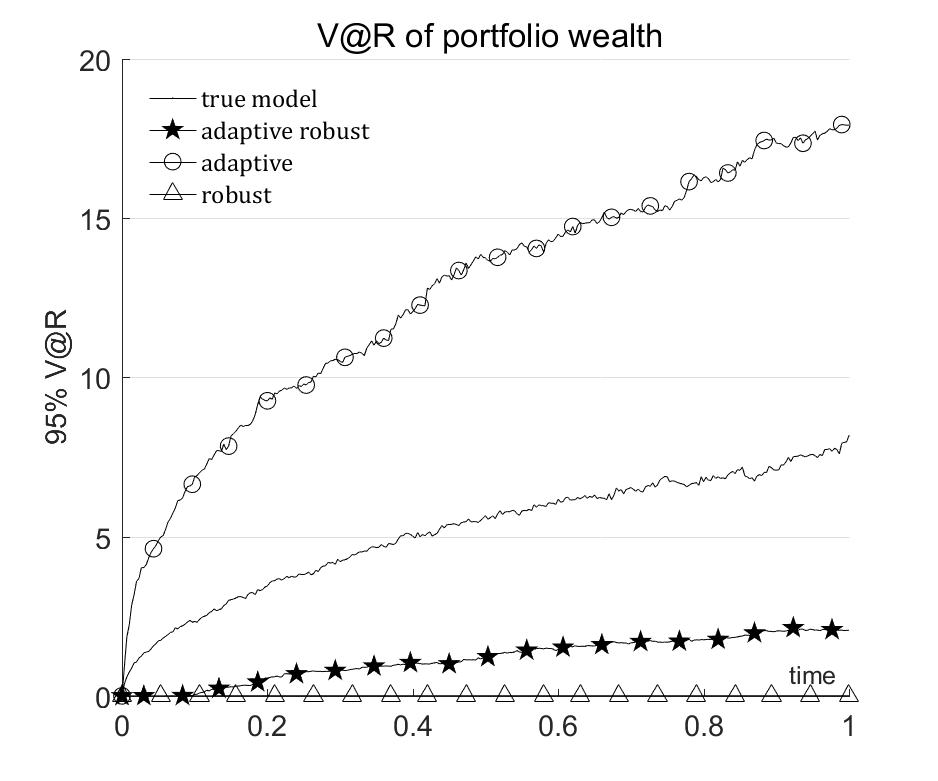}
    \end{subfigure}
    \hfill
    \begin{subfigure}[t]{0.48\textwidth}
        \centering
        \includegraphics[width=\linewidth]{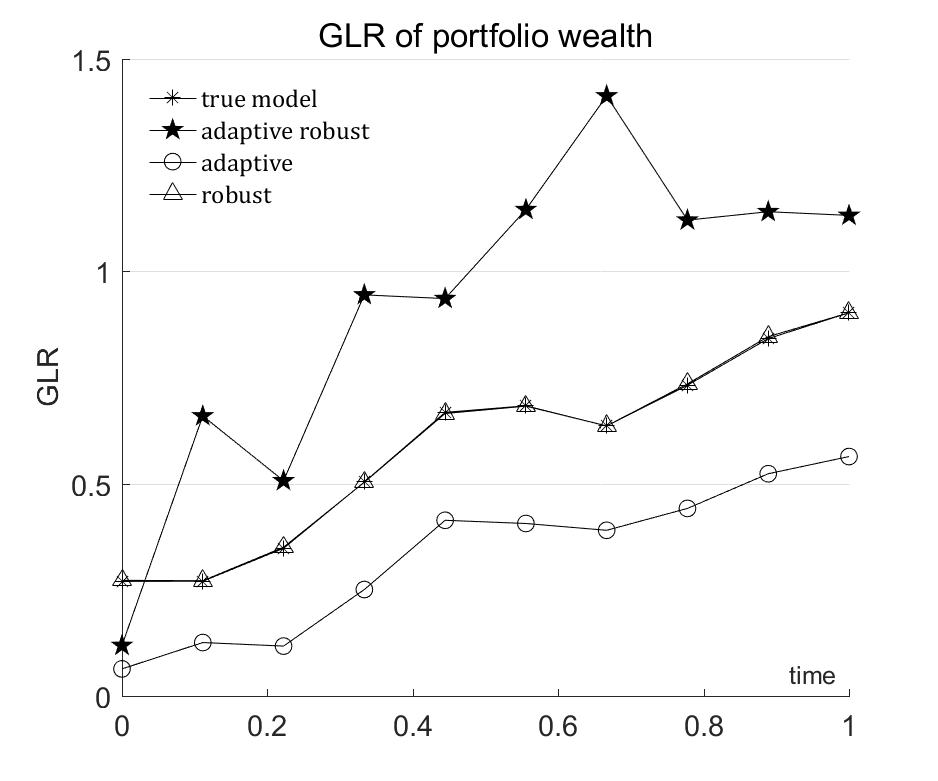}
    \end{subfigure}
\caption{Time-series of portfolio $\var$ and $\glr$. Unknown mean.}
\label{fig:PortValueVaRAndGLR-1d}
\end{figure}

 According to Figure~\ref{fig:PortValueVaRAndGLR-1d}, it is apparent that adaptive robust method has the best performance among the considered methods. GLR in case of  adaptive robust control  is higher than in case of classical robust control and than in case of adaptive control for all the terminal times except for $T=0.1$. Moreover, even though the adaptive control produces the highest mean terminal wealth (cf. Figure~\ref{fig:PortValueMeanAndStDev-1d}), the adaptive control is nevertheless the most risky method in the sense that the corresponding terminal wealth has the highest standard deviation and value at risk. The reason behind such phenomenon is, arguably,  that adaptive control method uses the point estimator while solving the optimization problem, so it can be overly aggressive and offers no protection against the estimation error which always exists.

Optimal portfolio wealth corresponding to classical robust method is the lowest among all the four approaches analyzed. This is not surprising since classical robust control  is designed to deal with the worst case scenario. Therefore, as illustrated by Figure~\ref{fig:portRiskyAssetAlloc-1d}, optimal holdings in the risky asset given by this method are always 0, which means that an investor following the classical robust method puts all the money in the banking account and, thus, fails to benefit from the price rise of the risky asset.

\begin{figure}[!ht]
\centering
  \includegraphics[scale=0.30]{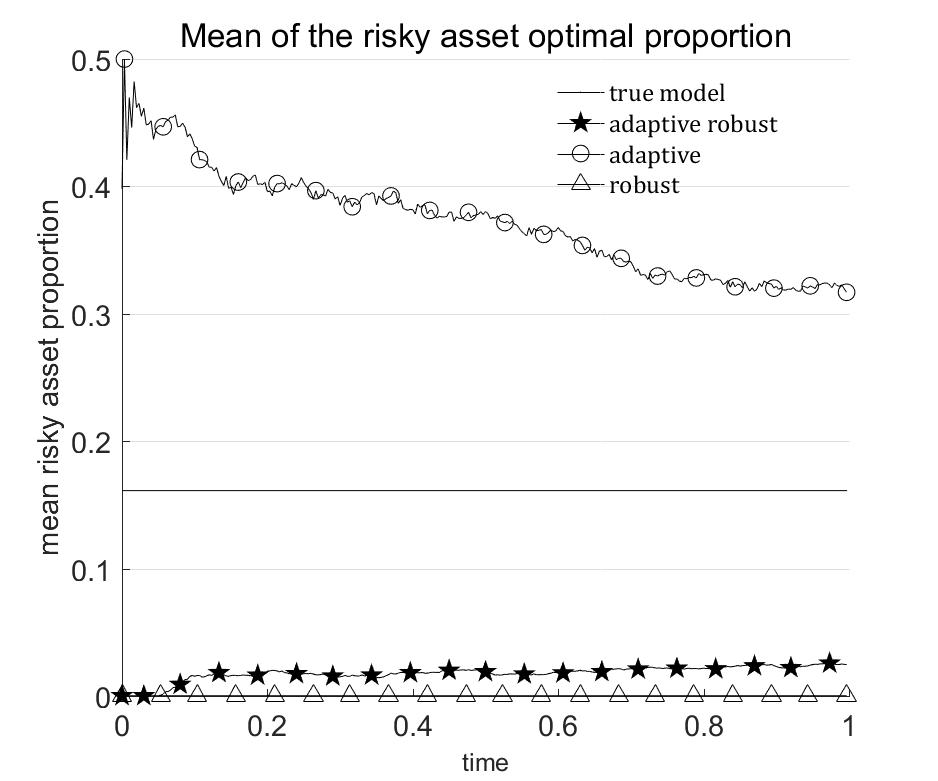}
  \caption{Time-series of optimal strategies means. Unknown mean.}
  \label{fig:portRiskyAssetAlloc-1d}
\end{figure}

Adaptive robust control method is meant to find the right balance between being aggressive and conservative.
As shown in Figures \ref{fig:PortValueMeanAndStDev-1d} and \ref{fig:PortValueVaRAndGLR-1d}, adaptive robust produces higher terminal wealth than classical robust, and it bears lower risk than adaptive control.
The robustness feature of the adaptive robust control method succeeds in controlling  the risk stemming from the model uncertainty. Moreover, the learning feature of this method prevents it from being too conservative.

\noindent\textbf{Case II.} Here, the adaptive robust control  method, the classical robust control method and the adaptive control method need to account uncertainty regarding the true parameter $\theta^*=(\mu^*,(\sigma^*)^2)$.

\begin{figure}[!ht]
\centering
  \includegraphics[scale=0.30]{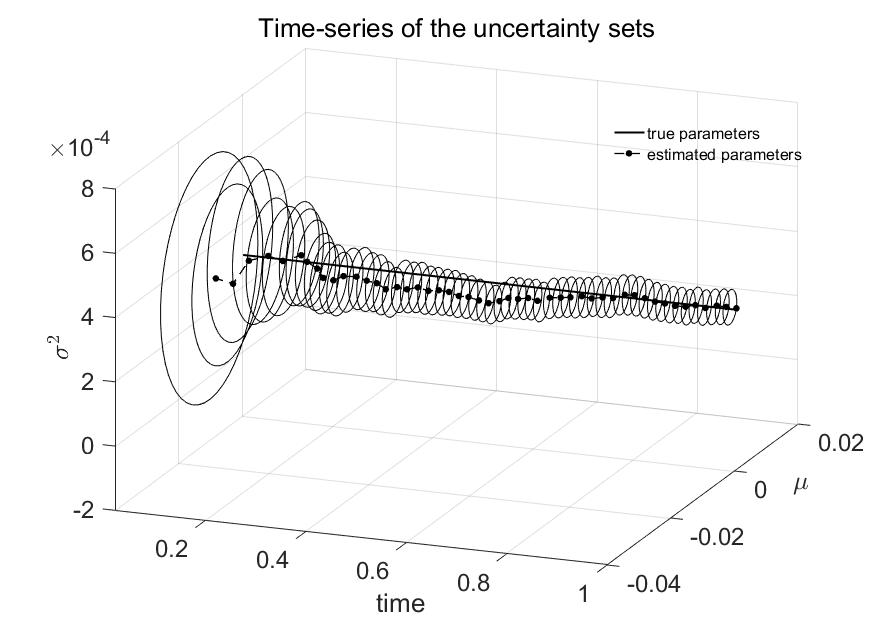}
\caption{Times-series of the confidence regions $\tau(t,{\widehat  \mu_{t}},{\widehat  \sigma^{2}_{t}} )$ for one particular path of $({\widehat  \mu_{t}},{\widehat  \sigma^{2}_{t}})$ at confidence level $\alpha=10\%$.}
\label{fig:TimeSeriesConfidenceRegion}
\end{figure}

We choose the parameter set as $\boldsymbol\Theta=[-1,1]\times[0,0.5]$.
As in Case I we consider a set of time horizons $T=0.1, 0.2, \ldots, 0.9, 1$, and $300T$ time iterations (or rebalancing dates) evenly distributed over the time horizon $T$.
Construction of the discretized state space (required in the adaptive control method) and application of the computed optimal strategies are made over $1000$ sample paths of the true model.
The other parameters are chosen as follows
\begin{eqnarray*}
    V_0=100,\quad r=0.02,\quad \alpha=0.1,\quad \gamma=20,\quad \mu^*=0.09,\quad \sigma^{*} = 0.30, \quad \widehat \mu_0=0.1,\quad \widehat \sigma_{0} = 0.4.
\end{eqnarray*}

Figure \ref{fig:TimeSeriesConfidenceRegion} shows the evolution of uncertainty sets for a particular sample path of the true model. We can show that the uncertainty on the true parameter $(\mu^{*}, \sigma^{*})$ quickly reduces in size as realized excess returns are observed through time. Moreover, we can notice that for $\alpha = 10\%$, the ellipsoid regions contains the true parameter for nearly  all time steps.

As in Case I we compared  performance of adaptive robust control, adaptive control, robust control and optimal control, by simulating paths generated by the true model.  Figure \ref{fig:PortValueMeanAndStDev-2d} shows how mean and standard deviation of optimal portfolio wealth values evolve through time. It also displays time-evolution of unexpected loss 95\%-VaR and Gain-to-Loss Ratio (GLR). In this 2-dimensional case, the conclusion are in line with the previous case I where only $\mu$ were assumed  unknown. The adaptive robust strategy outperforms adaptive control and robust control strategies in terms of GLR. Adaptive control strategy gives the highest portfolio wealth mean but at the cost of relatively high standard deviation. The unexpected loss 95\%-VaR series is also not in favor of the adaptive control approach.

Finally, we want to mention that optimal holdings in the risky asset in this case exhibit similar behaviour as in Case I.

\begin{figure}[!ht]
    \centering
    \begin{subfigure}[t]{0.48\textwidth}
        \centering
        \includegraphics[width=\linewidth]{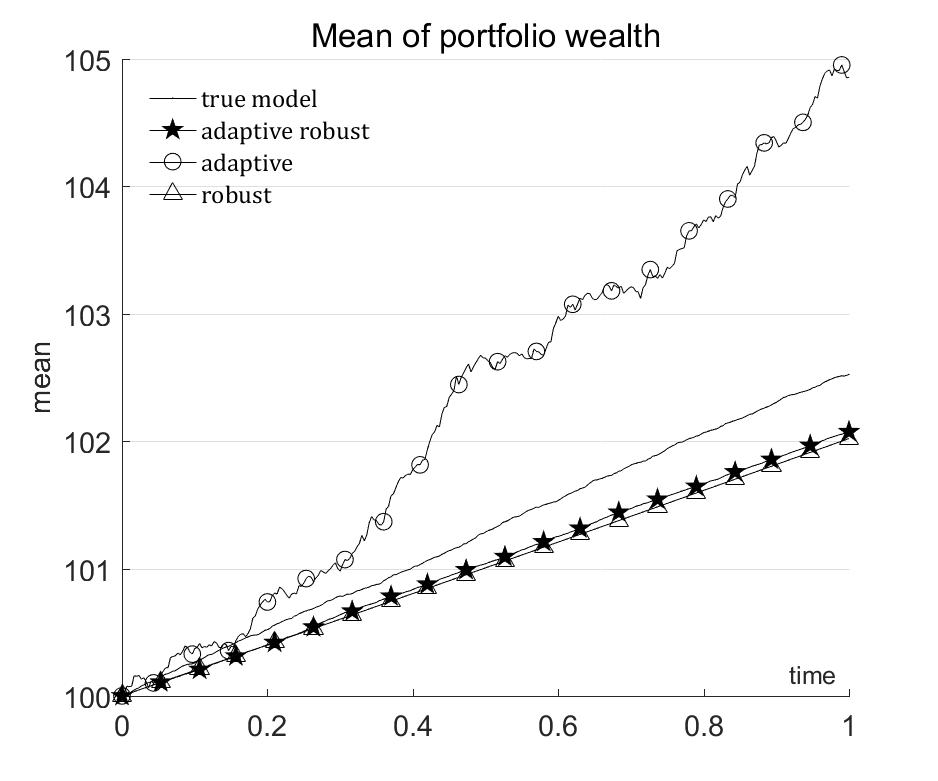}
    \end{subfigure}
    \hfill
    \begin{subfigure}[t]{0.48\textwidth}
        \centering
        \includegraphics[width=\linewidth]{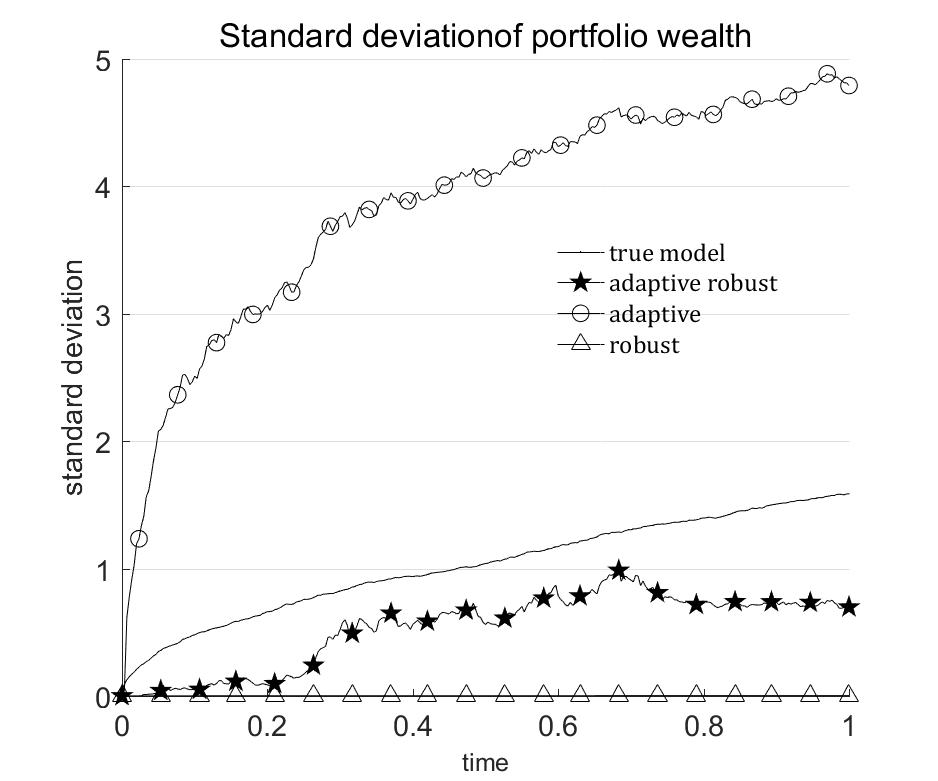}
    \end{subfigure}
    \caption{Time-series of portfolio wealth means, and standard deviations. Unknown mean and variance.}
    \label{fig:PortValueMeanAndStDev-2d}
\end{figure}

\begin{figure}
       \centering
    \begin{subfigure}[t]{0.48\textwidth}
        \centering
        \includegraphics[width=\linewidth]{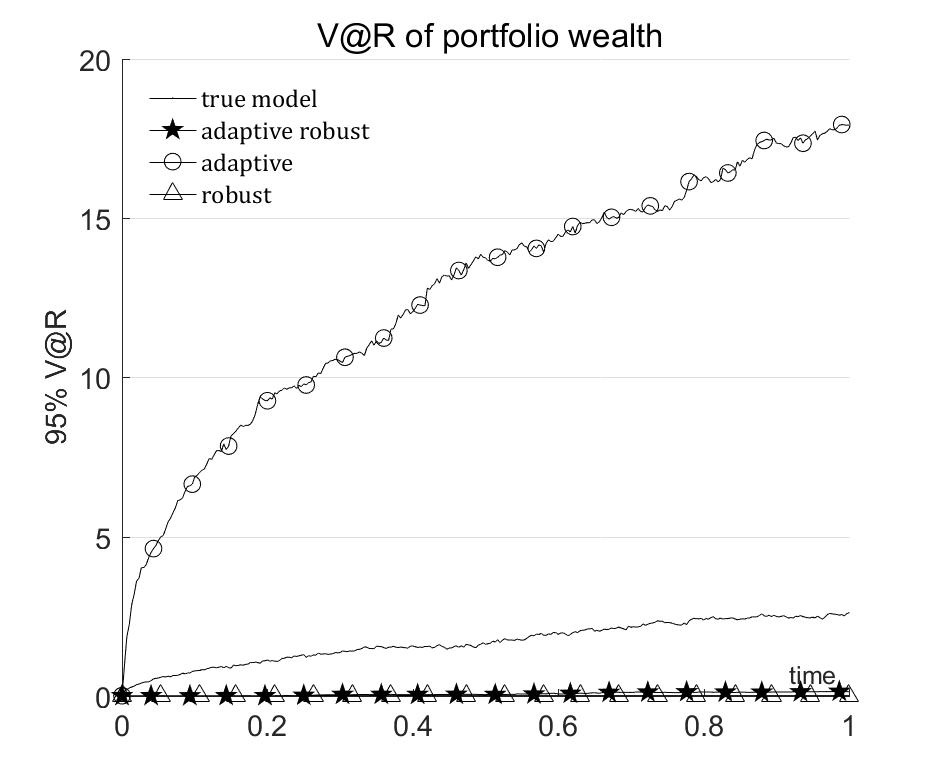}
    \end{subfigure}
    \hfill
    \begin{subfigure}[t]{0.48\textwidth}
        \centering
        \includegraphics[width=\linewidth]{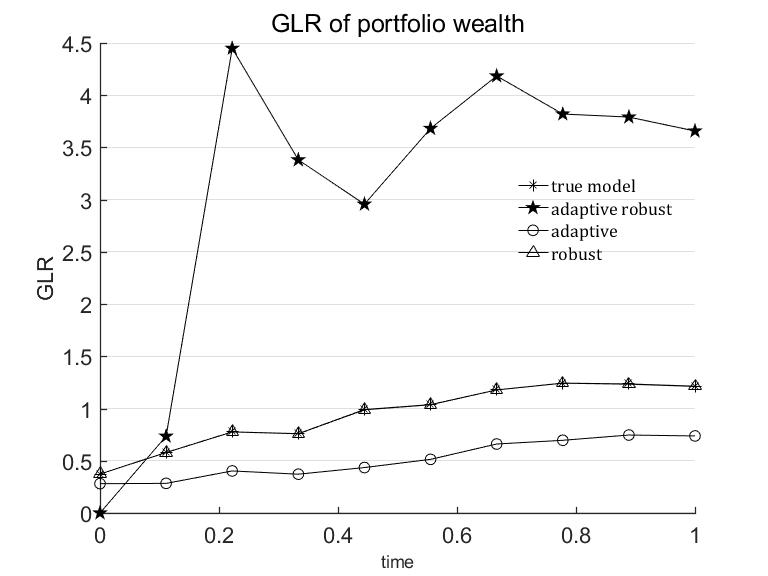}
    \end{subfigure}
\caption{Time-series of portfolio $\var$ and $\glr$. Unknown mean and variance.}
\label{fig:PortValueVaRAndGLR-2d}
\end{figure}

\section*{Acknowledgments}
Part of the research was performed while Igor Cialenco was visiting the Institute for Pure and Applied Mathematics (IPAM), which is supported by the National Science Foundation. The research of Monique Jeanblanc is  supported by  `Chaire Markets in transition', French Banking Federation  and ILB, Labex ANR 11-LABX-0019.

\bibliographystyle{alpha}

\end{document}